\documentclass{daj}

\dajAUTHORdetails{%
  title = {Quantitative twisted patterns in positive density subsets}, %% please capitalize all significant words
  author = {Kamil Bulinski and Alexander Fish},
  plaintextauthor = {Kamil Bulinski and Alexander Fish},
  plaintexttitle = {Quantitative twisted patterns in positive density subsets}, %%  title without math or LaTeX
  runningtitle = {Quantitative twisted patterns}, 
  runningauthor = {Kamil Bulinski and Alexander Fish},
  copyrightauthor = {Kamil Bulinski and Alexander Fish},
  keywords = {twisted recurrence, uniform bounds, sets of positive density},
}   %%% END \dajAUTHORdetails

\dajEDITORdetails{%
   year={2024},
   number={1},
   received={27 May 2022},   % received date: example: 7 January 2017
   published={30 April 2024},  % published date
   doi={10.19086/da.117029},       % XXX = number of paper, e.g. da006 for paper#6
}   %%% END \dajEDITORdetails

\usepackage[latin1]{inputenc}
\usepackage{tikz}
\usetikzlibrary{shapes,arrows}
\usetikzlibrary{arrows.meta}
\usepackage{forest}
\usepackage{tikz-qtree}

\usetikzlibrary{arrows,shapes,positioning,shadows,trees}

\usepackage{etoolbox}
\usepackage{amsmath}
\usepackage{enumerate}
\usepackage{amssymb}
\usepackage{amscd}
\usepackage{amsthm}
\usepackage{amsfonts}
\usepackage{graphicx}

\usepackage{float}
\usepackage[margin=3cm]{geometry}
\usepackage{bigints}
\usepackage{dsfont}

\newcommand{\bC}{\mathbb{C}}

\newcommand{\bN}{\mathbb{N}}

\newcommand{\bQ}{\mathbb{Q}}
\newcommand{\bR}{\mathbb{R}}
\newcommand{\bS}{\mathbb{S}}
\newcommand{\bT}{\mathbb{T}}

\newcommand{\bZ}{\mathbb{Z}}

\newcounter{dummy} \numberwithin{dummy}{section}

\theoremstyle{definition}
\newtheorem{mydef}[dummy]{Definition}
\newtheorem{prop}[dummy]{Proposition}

\newtheorem{thm}[dummy]{Theorem}
\newtheorem{lemma}[dummy]{Lemma}
\newtheorem{eg}[dummy]{Example}

\newtheorem{remark}[dummy]{Remark}

\newtheorem{theo}{Theorem}
\newtheorem{mainthm}{Theorem}[theo] % Use theo as driver counter

\begin{document}

\begin{frontmatter}[classification=text]
%% EDITOR: this will force the keywords to appear right after the Abstract.
%%   If the abstract is too long and would force the keywords off the
%%   front page, please comment out % [classification=text] above
%%   This way the keywords will be floated on the bottom of the first page
%%   even though the Abstract spills over to the next page.

%%% AUTHOR: Title goes here.  This line is optional.  You must use it
%%   if title has footnote attached or requires nontrivial typesetting,
%%   e.g., inclusion of linebreaks to force nice layout.
\title{Quantitative twisted patterns in positive density subsets}%\footnote{This is a footnote to the title}} %% please capitalize all significant words

%%% AUTHOR:
%%% List all authors. If you wish, place grant acknowledgements in \thanks.
%%% In brackets include a short tag for each author.
\author[kamil]{Kamil Bulinski\thanks{Supported by the Australian Research Council grant DP210100162.}}
\author[alexander]{Alexander Fish\thanks{Supported by the Australian Research Council grant DP210100162.}}

%%% AUTHOR: Abstract goes here
\begin{abstract}
We make quantitative improvements to recently obtained results on the structure of the image of a large difference set under certain quadratic forms and other homogeneuous polynomials. Previous proofs used deep results of Benoist-Quint on random walks in certain subgroups of $\operatorname{SL}_r(\bZ)$ (the symmetry groups of these quadratic forms) that were not of a quantitative nature. Our new observation relies on noticing that rather than studying random walks, one can obtain more quantitative results by considering polynomial orbits of these group actions that are not contained in cosets of submodules of $\bZ^r$ of small index. Our main new technical tool is a uniform Furstenberg-S\'{a}rk\"{o}zy theorem that holds for a large class of polynomials not necessarily vanishing at zero, which may be of independent interest and is derived from a density increment argument and Hua's bound on polynomial exponential sums. 
\end{abstract}
\end{frontmatter}

%%% AUTHOR: body of paper starts here
\section{Introduction}

We begin by recalling the following result of Magyar on the abundance of distances in positive density subsets of $\bZ^d$, for $d \geq 5$.

\begin{thm}[Magyar \cite{Magyar}] \label{Magyar thm} For all $\epsilon>0$ and integers $d \geq 5$ there exists a positive integer $k=k(\epsilon,d)>0$ such that the following holds: If $B \subset \bZ^d$ has upper Banach density $$d^*(B):= \lim_{L \to \infty} \max_{x \in \bZ^d} \frac{|B \cap (x+[0,L)^d)|}{L^d}> \epsilon,$$ then there exists a positive integer $N_0=N_0(B)$ such that $$ k\bZ_{>N_0} \subset \left\{ \| b_1 - b_2\|^2 \text{ } | \text{ } b_1,b_2 \in B \right\}.$$  \end{thm}

It was realised in a series of works initiated by Bj\"orklund and the authors (see \cite{BFChar}, \cite{BBTwisted}) that a similar result holds if one replaces the Euclidean squared distance $\| \cdot \|^2$ with other quadratic forms or other general functions. Before we state the general results, let us focus on some simple examples to demonstrate the type of questions of interest.

\begin{thm}[\cite{BFChar}] \label{thm: BF xy-z^2} Let $F:\bZ^3 \to \bZ$ be the non-positive definite quadratic form $F(x,y,z) = xy - z^2$ or $F(x,y,z) = x^2 - y^2 - z^2$. Then for all $B \subset \bZ^3$ of positive upper Banach density there exists a positive integer $k$ such that $$k\bZ \subset F(B-B).$$ \end{thm}

Note however that unlike for the case of $\| \cdot \|^2$ in Magyar's theorem, it was not established in these works that $k$ depends only on the upper Banach density of $B$. One of the main purposes of this paper is to demonstrate that this integer $k$ does indeed depend only on $d^*(B)$ and not on $B$, and thus answering affirmatively a question posed in \cite{FishProduct}.

\begin{thm}\label{thm: example of uniform virtually recurrent level sets} Let $F:\bZ^3 \to \bZ$ be the non-positive definite quadratic form $F(x,y,z) = xy - z^2$ or $F(x,y,z) = x^2 - y^2 - z^2$. Then for all $\epsilon > 0$ there exists a positive integer $k = k(\epsilon) \in \bZ_{>0}$ such that for all $B \subset \bZ^3$ with $d^*(B) > \epsilon$ we have that $$k\bZ \subset F(B-B).$$ \end{thm}

The  techniques that we develop in this paper allow us  to obtain quantitative polynomial extension of Bogolyubov's result \cite{Bogolyubov} on the linear image of a difference set.

\begin{thm}[Polynomial Bogolyubov's theorem] \label{thm: polynomial bog}
Let $\epsilon > 0$,  and assume that $R(n) \in \bZ[n]$ is a polynomial such that $\deg{(R)} \ge 2$ and $R(0) = 0$.  Let $E \subset \bZ^2$ with $d^*(E) > \epsilon$. Then there exists $k \le k(\epsilon, R)$ such that 
\[
k \bZ \subset \{ x + R(y) \, | \, (x,y) \in E-E \}.
\]
\end{thm}

Recall that a set $A \subset \bZ^r$ is called \textit{recurrent} or a \textit{set of recurrence} if for all measure preserving actions $T:\bZ^r \curvearrowright (X,\mu)$ and $B \subset X$ with $\mu(B)>0$ we have that $$\mu(T^aB \cap B) > 0 \quad \text{for some } a \in A.$$ By Furstenberg's correspondence principle \cite{Fur77}, if $A$ is recurrent then for all $B \subset \bZ^r$ with $d^*(B) > 0$ we have that $(B-B) \cap A \neq \emptyset$. For example, the Furstenberg-S\'{a}rk\"{o}zy theorem states the the set of squares is a set of recurrence. Thus it makes sense for us to make the following convenient definition.

\begin{mydef} We say that a function $F:\bZ^d \to \mathcal{S}$ has \textit{virtually recurrent level sets} if it satisfies the following condition: For all measure preserving actions $T:\bZ^d \curvearrowright (X,\mu)$ and $B \subset X$ with $\mu(B)>0$ there exists a positive integer $k$ such that for all $s \in F(k\bZ^d)$ there exists $v \in \bZ^d$ with $F(v) = s$ and $$\mu(B \cap T^{v}B) > 0. $$  We say that $F:\bZ^d \to \mathcal{S}$ has \textit{uniformly virtually recurrent level sets}  if for all  $\epsilon>0$ there exists a positive integer $k=k(\epsilon) > 0$ such that for all measure preserving actions $T:\bZ^d \curvearrowright (X,\mu)$ and $B \subset X$ with $\mu(B)>\epsilon$ we have that for all $s \in F(k\bZ^d)$ there exists a $v \in \bZ^d$ such that $F(v) = s$ and $$\mu(B \cap T^v B ) > 0. $$  \end{mydef}

By Furstenberg's correspondence principle we have the following combinatorial consequences of having virtually recurrent level sets.

\begin{prop} A function $F:\bZ^d \to \mathcal{S}$ has virtually recurrent level sets if for all sets $B \subset \bZ^d$ with $d^*(B)>0$ there exists a positive integer $k$ such that $F(k\bZ^d) \subset F(B-B)$. A function $F:\bZ^d \to \mathcal{S}$ has uniformly virtually recurrent level sets if for all $\epsilon>0$ there exists a positive integer $k=k(\epsilon) > 0$ such that $F(k \bZ^d) \subset F(B-B),$ for all $B \subset \bZ^d$ satisfiying $d^*(B) > \epsilon$. \end{prop}

Thus Theorem~\ref{thm: BF xy-z^2} is a consequence of the statement that the maps $F(x,y,z) = xy - z^2$ and $F(x,y,z) = x^2-y^2-z^2$ have virtually recurrent level sets, which was shown in \cite{BFChar} and \cite{BBTwisted}. While our new result Theorem~\ref{thm: example of uniform virtually recurrent level sets} is a consequence of the statement that these maps have uniformly virtually recurrent level sets.

\begin{remark} In order to be able to say that $\| \cdot \|^2$ has uniformly virtually recurrent level sets, one would need to remove a finite set of exceptions, i.e., replace the condition $F(k\bZ^d) \subset F(B-B)$ with $F(k\bZ^d) \subset F(B-B) \cup A_B$ where $A_B$ is a finite set (depending on $B$) as per Magyar's theorem (see the Ergodic formulation given by the first author in \cite{BulinskiSpherical}). However the examples that we study in this paper do not require this hence we avoid it.

\end{remark}

To motivate the use of this terminology note that if we could take $k=1$ then this would mean that for each $s$ in the range of $F$, the level set $\{ v \in \bZ^d | f(v) = s\}$ is recurrent. On the other hand, if $k=k(B) \neq 1$ then we can only say that those level sets with values in the image of a finite index subgroup of $\bZ^d$, namely $k\bZ^d$, are \textit{recurrent w.r.t to $B$} (note that $k$ could be different for each $B$ so it is not quite correct to say that those particular level sets are recurrent). But in geometric group theory the adverb \textit{virtually} means \textit{up to a finite index subgroup}.

\begin{remark} Every finite index subgroup $W \leq \bZ^d$ of index $k = |\bZ^d / W|$ contains $k'\bZ^d$ for some $k' \leq k^d$, so in fact this further justifies our use of the word virtual (i.e., we did not need to restrict our attention to just those subgroups of the form $k\bZ^d$, but it is convenient to do so).

\end{remark}

To provide an example of function which does not satisfy the conclusion of Theorem~\ref{thm: BF xy-z^2}, consider the linear function $F:\bZ^2 \to \bZ$ given by $F(x_1,x_2)=x_1+x_2$ and consider the Bohr set $B=B(\theta,\epsilon)\times B(\theta,\epsilon) \subset \bZ^2$, where $$B(\theta,\epsilon)=\{x \in \bZ ~|~ x\theta  \in (- \epsilon, \epsilon) \quad (\mathrm{mod} ~ 1) \},$$ for some irrational $\theta \in \bT=\bR/\bZ$ and small enough $\epsilon>0$. Then $d^{\ast}(B)=4\epsilon^2>0$ but $$F(B-B)=B(\theta,\epsilon) - B(\theta,\epsilon) + B(\theta,\epsilon) - B(\theta,\epsilon) \subset B(\theta,4\epsilon)$$ is contained in another Bohr set, which cannot contain a non-trivial subgroup of $\bZ$ whenever $\theta$ is irrational and $\epsilon<\frac{1}{8}$, as $k\bZ \theta$ is dense in $\bT$ for all non-zero integers $k$. On the other hand, a theorem of Bogolyubov \cite{Bogolyubov} (see also \cite{RuzsaBook}) states that for any $B \subset \bZ$ of positive upper Banach density, the set $B-B+B-B$ contains a Bohr set.

\subsection{Recurrent orbits}

The strategy in \cite{BFChar} for establishing that a function $F$ has virtually recurrent level sets involved studying the orbits of the linear automorphism group of $F$ and taking advantage of the fact that - for the functions $F$ of interest - this group has a rich enough algebraic structure (in particular, that action on $\bR^d$ is irreducible), which enabled the use of deep results of Benoist-Quint. This motivates us to make the following convenient definition.

\begin{mydef} A semigroup action $\Gamma \curvearrowright \bZ^d$ is said to have \textit{virtually recurrent orbits} if for all measure preserving actions $T:\bZ^d \curvearrowright (X,\mu)$ and $B \subset X$ with $\mu(B) > 0$ there exists a positive integer $k$ such that for all $v \in \bZ^d$ there exists a $\gamma \in \Gamma$ such that $$\mu( B \cap T^{ k \gamma v}B ) > 0 .$$ On the other hand, if $A \subset \bZ^d$ we say that a semigroup action  $\Gamma \curvearrowright \bZ^d$ has \textit{uniformly virtually recurrent orbits across $A$} if for each $\epsilon>0$ there exists a positive integer $k_0 = k_0(\epsilon)$ such that for all measure preserving actions $T:\bZ^d \curvearrowright (X,\mu)$ and $B \subset X$ with $\mu(B) > \epsilon$ there exists $0 < k \leq k_0(\epsilon)$ such that for each $v \in A$ there exists $\gamma \in \Gamma$ such that $$\mu(B \cap T^{k\gamma v}B) > 0.$$

\end{mydef}

To explain the relationships between these concepts, suppose that $\Gamma \leq \operatorname{GL}_d(\bZ) $ is the group of linear automorphisms of $F$, then for each fixed $v \in \bZ^d$ the orbit $k \Gamma v$ lies in the level set $F^{-1}(F(kv))$. Thus if $\Gamma \curvearrowright \bZ^d$ has virtually recurrent orbits then $F$ has virtually recurrent level sets. As for uniformly virtually recurrent orbits, there is the subtlety that $k$ is only uniformly bounded but not necessarily the same for all $B$ with $\mu(B)>\epsilon$, nonetheless we still recover a constant $k$ in the definition of uniformly virtual recurrent level sets by considering $k_0!$ (as $F(k_0! \bZ^d) \subset F(k\bZ^d)$ for all $k \leq k_0$). Unfortunately, for the examples of interest we have not been able to establish uniform virtual recurrence of all orbits, however we will show that we can for at least one orbit in each level set, which implies virtually recurrent level sets (each level set may be a disjoint union of many different orbits that we cannot all control).

Using this language, we can restate the main theorem in \cite{BFChar} as follows.

\begin{thm}[\cite{BFChar}] \label{thm: BFChar on adjoint representation} Let $\mathfrak{sl}_n(\bZ) \cong \bZ^{n^2-1}$ denote the additive group of $n \times n$ integer matrices of trace zero. Let $\Gamma = \operatorname{SL}_n(\bZ)$ act on  $\mathfrak{sl}_n(\bZ)$ by conjugation (the adjoint representation). Then this action has virtually recurrent orbits. In particular, as conjugation preserves determinants, the characteristic polynomial map $\mathfrak{sl}_n(\bZ) \to \bZ[t]$ given by $C \mapsto \det(tI - C)$ has virtually recurrent level sets.
\end{thm}

In particular, Theorem~\ref{thm: BFChar on adjoint representation} implies the statement of Theorem~\ref{thm: BF xy-z^2}  for $F(x,y,z) = xy - z^2$. Indeed, the map $F:\bZ^3 \to \bZ$ given by $F(x,y,z) = xy-z^2$ can be seen as the determinant of the $2\times2$ matrix obtained by identifying $(x,y,z) \in \bZ^3$ with $$ \begin{bmatrix}
    z   & -y \\
    x  & -z \\
\end{bmatrix} \in \mathfrak{sl}_2(\bZ).$$ Since the determinant map has virtually recurrent level sets, for any set $B \subset \bZ^3$ there exists $k$ with $F(k\bZ^3) \subset F(B-B)$. But $F(k\bZ^3) = k^2\bZ$ and this proves the statement of Theorem~\ref{thm: BF xy-z^2} for $F(x,y,z) = xy-z^2$. We defer the proof of the remaining part of Theorem~\ref{thm: BF xy-z^2} regarding the quadratic form $F(x,y,z) = x^2-y^2-z^2$ to Section~\ref{Section7}.

We now provide a quantitative improvement of this result by demonstrating that this characteristic polynomial map is uniformly virtually recurrent, answering a question raised in \cite{FishProduct}. To do this, we show that this action of $SL_n(\bZ)$ by conjugation has uniformly virtually recurrent orbits across the companion matrices of characteristic polynomials.

\begin{mydef} Given a polynomial $p(t) = a_0  + a_1t + a_2t^2 + \cdots + a_{n-1}t^{n-1} + t^n$, we can define its companion matrix by

$$c_p = \begin{bmatrix}
    0   & \dots & 0 & 0 & -a_0 \\
    1  & \dots & 0 & 0 & -a_1 \\
    \vdots & \ddots & \vdots & \vdots & \vdots \\
    0 & \dots & 1  & 0 & -a_{n-2} \\
    0 & \dots & 0 & 1  & -a_{n-1}
\end{bmatrix}$$

\end{mydef}

\begin{mainthm} \label{thm: uniform virtual recurrence for adjoint representation} Let $\Gamma = \operatorname{SL}_n(\bZ)$ act on  $\mathfrak{sl}_n(\bZ)$ by conjugation (the adjoint representation). Let $$ A = \{ c_p ~|~ p(t) \in \bZ[t] \text{ with } p^{(n-1)}(0) = 0 \text{ and } \deg{p} = n\}$$ be the set of companion matrices of integer polynomials with zero $(n-1)$-st term (so that the corresponding companion matrix has trace $0$ and thus in $\mathfrak{sl}_n(\bZ)$). Then this conjugation action has uniformly recurrent orbits across $A$. 
\end{mainthm}

Since the image of the characteristic polynomial  map $C \mapsto \det(tI - C)$ on the set $A$ in Theorem~\ref{thm: uniform virtual recurrence for adjoint representation} covers all  characteristic polynomials of the matrices in $\mathfrak{sl}_n(\bZ)$,  it follows %from Theorem~\ref{thm: uniform virtual recurrence for adjoint representation} 
that  the characteristic polynomial map  has \textbf{uniformly} virtually recurrent level sets. Likewise, every coefficient of the characteristic polynomial also has uniformly recurrent level sets. By specifying to the free coefficients of the characteristic polynomials  (the determinant) in the dimension $n = 2$, the quadratic form $$F(x,y,z) = xy - z^2$$ has \textbf{uniformly} virtually recurrent level sets. This  implies Theorem~\ref{thm: example of uniform virtually recurrent level sets} for the quadratic form $F(x,y,z) = xy - z^2$. The second part of Theorem~\ref{thm: example of uniform virtually recurrent level sets} is proved in Section~\ref{Section7}.

We are able to recover this result by demonstrating that if the group is generated by unipotents and the orbit is not contained in any proper affine subspace (this holds for non-zero vectors and irreducible representations in dimension greater than $1$), then the orbit is virtually recurrent.

\begin{mydef} A set $S \subset \bR^r$ is said to be \textit{hyperplane-fleeing} if for all proper affine subspaces $H$ of $\bR^d$ (i.e., $H = W + a$ for some proper vector subspace $W \subset \bR^r$ and $a \in \bR^r$) we have that $S \not\subset H$.

\end{mydef}

\begin{thm} Suppose that $\Gamma \leq \operatorname{GL}_r(\bZ)$ is a subgroup generated by a finite set of unipotents such that for each non-zero $v \in \bZ^r$ the orbit $\Gamma v$ is hyperplane-fleeing (this holds if $r>1$ and the linear action $\Gamma \curvearrowright \bR^r$ is irreducible). Then this action has virtually recurrent orbits.

\end{thm}

Our next main result provides a quantitative improvement by demonstrating that if the orbit also flees cosets of subgroups of large index sufficiently quickly, then we have \textbf{uniform} virtual recurrence.

\begin{mydef} Let $\Lambda$ be an abelian group. We say that $S \subset \Lambda$ is \textit{$Q$-coset fleeing (in $\Lambda$)} if for all subgroups $W \leq \Lambda$  with index $|\Lambda/W| > Q$ we have that $S$ is not contained in a coset $a+W$ of $W$, i.e., the image of $S$ in the quotient $\Lambda/W$ contains at least two elements.
\end{mydef}

\begin{mainthm} \label{thm: uniformly virtually recurrent orbits for unipotents} Let $\epsilon>0$ and $r, N, Q$ be positive integers. Then there exists a positive integer $k_0 = k_0(r, N, Q, \epsilon)$ such that for all measure preserving systems $T:\bZ^r \curvearrowright (X,\mu)$ and $B \subset X$ with $\mu(B) > \epsilon$ there exists
 $0 < k \leq k_0$ such that the following holds: Suppose that there exist $N$ unipotent elements $u_1, \ldots, u_N \in \operatorname{SL}_r(\bZ)$ and $v \in \bZ^r$ such that $$\mathcal{S} =  \{u_1^{n_1} \cdots u_N^{n_N} ~|~ n_1, \ldots, n_N \in \bZ \}$$ satisfies the property that $\mathcal{S} v$ is $Q$-coset fleeing in $\bZ^r$ and $\mathcal{S}v$ is also hyperplane-fleeing. Then $$\mu(B \cap T^{k\gamma v}B) > 0 \quad \text{for some } \gamma \in \mathcal{S}.$$
\end{mainthm}

Thus, the results stated above on the uniformity for the adjoint representation will follow from establishing uniform bounds on $N$ and $Q$ for the orbits of the companion matrices.

\subsection{A uniform Furstenberg-S\'{a}rk\"{o}zy theorem}

Our main technical tool, which may be of independent interest, is a quantitative uniform Furstenberg-S\'{a}rk\"{o}zy theorem that works for a large family of polynomials which do not necessarily have zero constant term.

\begin{mydef} We say that a vector of integer coefficient polynomials $P(n) = (P_1(n), \ldots, P_r(n))$ where $P_i(n) \in \bZ[n]$ has \textit{multiplicative complexity} $Q$ if for all $\vec{a} = (a_1, \ldots, a_r) \in \bZ^r$ and $q \in \bZ$ with $\text{gcd}(a_1, \ldots, a_r, q) = 1$ we have that the polynomial $$\sum_{j=1}^D b_j n^j = (P(n) - P(0)) \cdot \vec{a}$$ satisfies $\text{gcd}(b_1, \ldots, b_D, q) \leq Q$.

\end{mydef}

\begin{remark} For the case $r=1$, a polynomial $P(n) = \sum_{j=0}^D c_j D^j$ has multiplicative complexity $$\text{gcd}(c_1, \ldots, c_D).$$ \end{remark}

\begin{mainthm} \label{thm: uniform polynomial recurrence} Let $D,r , Q$ be positive integers and $\epsilon > 0$. There exists a positive integer $k_0 = k_0(D, r, Q, \epsilon)$ such that the following is true:  Let $T:\bZ^r \curvearrowright (X,\mu)$ be an ergodic measure preserving system and suppose that $B \subset X$ with $\mu(B) > \epsilon$. Then there exists a positive integer $k = k(B) \leq k_0$ such that whenever $P:\bZ \to \bZ^r$ is a polynomial $P(n) = (P_1(n), \ldots, P_r(n))$ with degree at most $D$ (that is, $P_i(n) \in \bZ[n]$ with $deg(P_i) \leq D$) such that $P(n)$ has multiplicative complexity $Q$ and $P$ is hyperplane-fleeing, then there exists arbitrarily large $n \in \bZ$ such that $$\mu(T^{kP(n)}B \cap B) > 0.$$

\end{mainthm}

Note that if we were to add the condition $P(0) = 0$ or just the intersectivity of the polynomial map $P$, i.e., for every $k \ge 1$ there exists $n \ge 1$ with $P(n) \in k\bZ^d$,  then the result holds true with $k=1$. The case $P(0) = 0$ is a classical the  Furstenberg-S\'{a}rk\"{o}zy theorem,  the intersective polynomial case for one variable has been proved by Kamae and Mend\'es-France \cite{KM}, while the multi-variable case was done in \cite{BLL}.

Moreover, if we were allowed to change $k$ (and even insist the bound $k \leq \epsilon^{-1} + 1$) for different polynomials $P$ then the result trivially follows from the Poincar\'e Recurrence theorem. We cannot afford these relaxations in the proof of Theorem~\ref{thm: uniformly virtually recurrent orbits for unipotents} from Theorem~\ref{thm: uniform polynomial recurrence} given in Section~\ref{section: from polynomials to unipotent actions}.

\begin{remark} We now demonstrate that we can not remove in Theorem~\ref{thm: uniform polynomial recurrence} the assumption that $P(n)$ has multiplicative complexity $Q$. To show this, suppose for contradiction that we could. Thus if we fix a positive integer $m \geq 3$ and choose $0 < \epsilon < \frac{1}{m}$ then this would mean that there exists a positive integer $k_0$ such that for every measure preserving system $T: \bZ \curvearrowright (X,\mu)$ we would have that whenever $B \subset X$ with $\mu(B) = \frac{1}{m}$ we would have a positive integer $k = k(B) \leq k_0$ such that for all integers $a_0, a_1$ with $a_1 \neq 0$ we would have $$\mu(B \cap T^{k(a_1n + a_0)}B) > 0 \quad \text{for some } n \in \bZ.$$ That is, we have chosen to focus on $\bZ$ systems and polynomials of degree $1$. Let $X = \bZ / mk_0 \bZ$ with uniform probability measure (Haar measure) and let $Tx = x + 1$. Now let $B = \{0, \ldots, k_0 - 1\} \subset X$, which has measure exactly $\frac{1}{m}$. Now choose $a_1 = k_0m$ and thus for all integers $n$ and $0 < k \leq k_0$ we have $$T^{k(a_1n + a_0)} B = T^{ka_0}B.$$ But since $|X| \geq 3k$, we can choose a suitable $a_0$ such that $T^{ka_0}B$ is disjoint from $B$, which is a contradiction.
\end{remark}
%\subsection{A polynomial Bogolyubov Theorem}

\section{Polynomial exponential bounds}

Throughout this paper, we let $e(t) = \exp(2\pi i t)$. We begin with a classical bound of Hua.

\begin{thm}[\cite{Hua}, see also \cite{HuaBook}]For $\epsilon > 0$ and positive integers $d$ there exists a constant $C_{d, \epsilon}$ such that if $f = a_0 + a_1 x + \cdots + a_d x^d \in \bZ[x]$ is a polynomial and $q$ is a positive integer such that $gcd(a_1, \ldots, a_d ,q) = 1$ then $$\left| \frac{1}{q} \sum_{n=1}^q e \left(\frac{f(n)}{q} \right) \right| \leq C_{d, \epsilon} q^{\epsilon - \frac{1}{d}}. $$
\end{thm}

We deduce a straightforward higher dimensional generalization that will be useful for us.

We let $\bS^1 = \{ z \in \bC ~|~ |z|=1 \}$ be the multiplicative group of unit complex numbers. A character is a group homomorphism $ \chi:\bZ^r \to \bS^1, $ i.e., $\chi(x + y) = \chi(x) \chi(y)$. Note that the image of $\chi$ has exactly $q$ elements if and only if it is of the form $$\chi(x_1, \ldots, x_r) = e\left(\frac{1}{q}(a_1x_1 + \cdots + a_rx_r) \right)$$ where $gcd(a_1, \ldots a_r,q ) = 1$.

\begin{prop} \label{prop: multivariate Hua bound} For $r, D, Q>0$ there exists a function $\epsilon_{r,D, Q}: \bN \to \bR$ with $\lim_{q \to \infty} \epsilon_{r,D, Q} (q) = 0$ such that the following is true: Suppose that $P: \bZ \to \bZ^r$ is a polynomial function of degree at most $D$, more precisely $P(n) = (P_1(n), \ldots, P_r(n))$ with $P_i(n) \in \bZ[n]$ with $\text{deg}(P_i) \leq D$. Suppose that $P(n)$ has multiplicative complexity $Q$. Then for all positive integers $q$ and characters $\chi:\bZ^r \to \bS^1$ with image of cardinality at least $q$ we have that $$\left| \frac{1}{q} \sum_{n=1}^q \chi (P(n)) \right| \leq \epsilon_{r,D, Q} (q).$$

\end{prop}

\begin{proof} Suppose that $\chi$ has cardinality exactly $q$. Hence $$\chi(x_1, \ldots, x_r) = e\left(\frac{1}{q}(a_1x_1 + \cdots + a_rx_r) \right)$$ where $gcd(a_1, \ldots a_r,q ) = 1$. Let $\vec{a} = (a_1, \ldots, a_r)$ and write $$P(n) \cdot \vec{a} = P(0) \cdot \vec{a} + \sum_{j=1}^{D'} b_j n^j$$ where $D' \leq D$. Let $q' = gcd(b_1, \ldots, b_{D'}, q) \leq Q$. Then by the Hua bound we have that

\begin{align*} \left| \frac{1}{q} \sum_{n=1}^q \chi (P(n)) \right| &= \left| \frac{1}{q} \sum_{n=1}^q e\left( \frac{1}{q} P(n) \cdot \vec{a} \right) \right| \\ 
&= \left| \frac{1}{q} \sum_{n=1}^q e\left( \frac{1}{q} \sum_{j=1}^{D'} b_j n^j \right)  \right| \\
&= \left| \frac{1}{q} \sum_{n=1}^q e\left( \frac{1}{q/q'} \sum_{j=1}^{D'} \frac{b_j}{q'} n^j \right)  \right| \\
&\leq C_{D', \epsilon} (q/q')^{\epsilon - \frac{1}{D'}} \\
&\leq C_{D, \epsilon} (q/Q)^{\epsilon - \frac{1}{D}}. \end{align*}

 \end{proof}

\section{$(q,\delta)$-Equidistributed sets.}

In this section, we introduce the notion of a $(q,\delta)$-equidistributed subset of an ergodic system, which was initially developed by the first author in \cite{BulinskiSpherical} in order to obtain a quantitative ergodic version of Magyar's theorem. Initially, the notion of $(q,\delta)$-equidistributed subsets of $\bZ^d$ was introduced by Lyall and Magyar in \cite{LM}. We include it again for the sake of completeness.

For the remainder of this section, let $F_N=[1,N]^d \cap \bZ^d$.

\begin{mydef} \label{def: equidistributed} Let $T: \bZ^d \curvearrowright (X,\mu)$ be an ergodic measure preserving action. Then we say that $B \subset X$ is \textit{$(q,\delta)$-equidistributed} if for almost all $x \in X$ we have $$ \lim_{n \to \infty} \frac{1}{|F_n|}\left | \{a \in F_n \text{ } | \text{ } T^{qa}x \in B \}\right| \leq (1+\delta)\mu(B). $$

\end{mydef}

\begin{mydef}[Conditional probability and ergodic components] \label{def: ergodic components of finite index} If $(X,\mu)$ is a probability space and $C \subset X$ is measurable with $\mu(C)>0$ then we define the conditional probability measure $\mu(\cdot | C)$ given by $\mu(B|C)=\frac{\mu(B \cap C)}{\mu(C)}$. We note that if $C$ is invariant under some measure preserving action, then $\mu(\cdot |C)$ is also preserved by this action. If $T: \bZ^d \curvearrowright (X,\mu)$ is ergodic and $k$ is a positive integer, then the action $T^k: \bZ^d \curvearrowright(X,\mu)$ may not be ergodic; but it is easy to see that there exists a $T^k$-invariant subset $C \subset X$ such that the action of $T^k$ on $C$ is ergodic (more precisely, $\mu( \cdot | C)$ is $T^k$-ergodic) and the translates of $C$ disjointly cover $X$ (there are at most $k^d$ distinct translates, hence $\mu(C) \geq k^{-d}$). Note that the translates of $C$ also satisfy these properties of $C$. We call such a measure $\mu( \cdot | C)$ a $T^k$-ergodic component of $\mu$. It follows that $\mu$ is the average of its distinct $T^k$-ergodic components.

\end{mydef}

We may now introduce our measure increment technique, which will be used to reduce our recurrence theorems to ones which assume sufficient equidistribution.

\begin{lemma}[Ergodic measure increment argument]\label{lemma: measure increment} Let $\delta,\epsilon>0$, let $q$ be a positive integer and let $T: \bZ^d \curvearrowright (X,\mu)$ be ergodic. If $B \subset X$ with $\mu(B)>\epsilon$ then there exists a positive integer $k  \leq q^{\log(\epsilon^{-1})/\log(1+\delta)}$ and a $T^k$-ergodic component, say $\nu$, of $\mu$ such that $\nu(B) \geq \mu(B)$ and $B$ is $(q,\delta)$-equidistributed with respect to $T^k:\bZ^d \curvearrowright (X,\nu)$. 

\end{lemma}

To study the limits appearing in Definition~\ref{def: equidistributed} we make use of the well known Pointwise Ergodic Theorem.

\begin{prop}[Pointwise Ergodic Theorem] Let $T: \bZ^d \curvearrowright (X,\mu)$ be a measure preserving action. Then for all $f \in L^2(X,\mu)$ there exists $X_f \subset X$ with $\mu(X_f)=1$ such that $$ \lim_{N \to \infty} \frac{1}{|F_N|} \sum_{a \in F_N} f(T^a x) \to P_T f (x) $$ for all  $x \in X_f$.

\end{prop}

\begin{proof}[Proof of Lemma~\ref{lemma: measure increment}] If $B$ is $(q,\delta)$ equidistributed, then we are done. Otherwise, it follows from the Pointwise Ergodic Theorem (applied to the action $T^q$ and the indicator function of $B$) that there exists a $T^q$-ergodic component of $\mu$, say $\nu_1$, such that $\nu_1(B) \geq (1+\delta)\mu(B)$. Continuing in this fashion, we may produce a maximal sequence of ergodic components $\nu_1,\nu_2 \ldots, \nu_J $ of $T^q,T^{q^2}, \ldots T^{q^J}$, respectively, such that $\nu_{j+1}(B) \geq (1+\delta) \nu_{j} (B)$. Clearly we must have $\epsilon(1+\delta)^J \leq 1$ and so this finishes the proof with $k=q^J$. \end{proof}

We now turn to demonstrating the key spectral properties of a $(q,\delta)$-equidistribution set.

\begin{mydef}[Eigenspaces] \label{def: eigen} If $T: \bZ^d \curvearrowright (X,\mu)$ is a measure preserving action and $\chi \in \widehat{\bZ^d}$ is a character on $\bZ^d$, then we say that $f \in L^2(X,\mu)$ is a $\chi$-eigenfunction if $$T^a f= \chi(a)f \text{ for all } a \in \bZ^d.$$ We let $\text{Eig}_T(\chi)$ denote the space of $\chi$-eigenfunctions and for $R \subset  \widehat{\bZ^d}$ we let $$\text{Eig}_T(R)=\overline{\text{Span}\{ f \text{ } | \text{ } f \in \text{Eig}(\chi) \text{ for some } \chi \in R \}}^{L^2(X,\mu)}.$$
In particular, we will be intersted in the sets $R_q = \{ \chi \in \widehat{\bZ^d}  \text{ } |\text{ } \chi^q=1 \}$ and $R^*_q= R_q \setminus \{ 1 \}$, where $q \in \bZ$.  Note that the spaces $\text{Eig}_T(\chi)$ are orthogonal to each other and hence $\text{Eig}_T(R)$ has an orthonormal basis consiting of $\chi$-eigenfunctions, for $\chi \in R$. Note also that Ergodicity implies that each $\text{Eig}_T(\chi)$ is at most one dimensional.\end{mydef}

\begin{prop} \label{prop: equidist implies small proj}  Let $T: \bZ^d \curvearrowright (X,\mu)$ be an ergodic measure preserving action and suppose that $B \subset X$ is $(q,\delta)$-equidistributed. Let $h \in L^2(X,\mu)$ be the orthogonal projection of $\mathds{1}_B$ onto $\text{Eig}_T(R^*_q)$. Then $$P_{T^q} \mathds{1}_B=\mu(B) + h$$ and $$\| h \|_2 \leq \sqrt{(2\delta + \delta^2)}\mu(B).  $$

\end{prop}

\begin{proof} Note that\footnote{This follows from the fact that all finite dimensional representations of a finite abelian group can be decomposed into one dimensional representations.} $\text{Eig}_T(R_q)=L^2(X,\mu)^{T^q}$. This, together with the ergodicity of $T$, shows that \\$h = P_{T^q}\mathds{1}_B - \mu(B)$. Now the pointwise ergodic theorem, applied to the action $T^q$, combined with the $(q,\delta)$-equidistribution of $B$ immediately gives that $$\| h \|^2_2 = \|P_{T^q} \mathds{1}_B \|^2 - \| \mu(B) \|^2_2 \leq (1+\delta)^2\mu(B)^2 - \mu(B)^2= (2\delta + \delta^2) \mu(B)^2.$$ \end{proof} 

\section{A quantitative polynomial mean ergodic theorem and proof of Theorem~\ref{thm: uniform polynomial recurrence}}

\begin{thm} \label{polynomial mean ergodic thm}

Let $D,r , Q$ be positive integers and $\epsilon > 0$. There exists a positive integer $q = q(D, r, Q, \epsilon)$ such that the following is true: Let $P:\bZ \to \bZ^r$ be a polynomial $P(n) = (P_1(n), \ldots, P_r(n))$ with degree at most $D$ (that is, $P_i(n) \in \bZ[n]$ with $deg(P_i) \leq D$) such that $P(n)$ has multiplicative complexity $Q$ and $P$ is hyperplane-fleeing. Let $T:\bZ^r \curvearrowright (X,\mu)$ be an ergodic measure preserving system and suppose that $B \subset X$ is $(q,\delta)$-equidistributed. Then \begin{align}\label{polynomial mean estimate} \limsup_{N \to \infty}\left \| \mu(B) - \frac{1}{N} \sum_{n=1}^N T^{P(n)} \mathds{1}_B\right \|_2 \leq \sqrt{3\delta} + \epsilon. \end{align}

\end{thm}

\begin{remark}
We could find an explicit  dependence of $q$ on $D, r, Q, \epsilon$ in Theorem~\ref{polynomial mean ergodic thm} but decided to refrain from this task since we are not looking for  optimal bounds on $k_0$ in Theorem~\ref{thm: uniform polynomial recurrence}. 
\end{remark}
To prove this theorem, we first show the following.

\begin{lemma} \label{lemma: polynomial average zero on irrational spectrum}  Suppose that $T: \bZ^d \curvearrowright (X,\mu)$ is a measure preserving system. Let $f \in L^2(X,\mu)$ be orthogonal to the rational Kronecker factor of $(X,\mu,T)$. Then for all polynomials $P_1(n), \ldots P_d(n) \in \bZ[n]$ such that no non-trivial $\bR$-linear combination of them is constant we have that $$\lim_{N \to \infty}\left \| \frac{1}{N} \sum_{n=1}^N T^{P(n)} f \right \|_2 =0,$$ where $P(n)=(P_1(n), \ldots, P_d(n)).$ 

\end{lemma}

\begin{proof} By the spectral theorem, there exists a positive Borel measure $\sigma$ on $\bT^d$ such that $$\langle T^v f, f \rangle = \int_{\bT^d} e(\langle v, \theta \rangle ) d\sigma(\theta) \quad \text{for all }v \in \bZ^d. $$ In particular, for each character $\chi$, written as $\chi(v) = e(\langle v, \theta \rangle)$ for some $\theta \in \bT^d$, we have that $$\langle P_{\text{Eig}_T}(\chi) f, f \rangle = \sigma(\{ \theta \})$$ where $P_{\text{Eig}_T}(\chi) f$ denotes the orthogonal projection to $f$ onto the $\chi$-eigenfunctions (this follows from applying the Mean Ergodic Theorem to the unitary action $\chi(v)^{-1}T^v$ since the $\chi(v)^{-1}T^v$ invariant functions are precisely the $\chi$-eigenfunctions). Hence since $f \in L^2_{\textbf{Rat}}(X,\mu,T)^{\perp}$ we have that $\sigma ( \bQ^d/\bZ^d)=0$ and hence $$\left \|\frac{1}{N} \sum_{n=1}^N T^{P(n)} f \right \|^2_2 =\int_{\Omega} \left| \frac{1}{N} \sum_{n=1}^N e( \langle  P(n), \theta\rangle ) \right|^2 d\sigma(\theta)  $$ where $\Omega=\bT^d \setminus (\bQ^d /\bZ^d).$ We now claim that if $\theta \in \Omega$ then the polynomial $\langle P(n), \theta \rangle \notin \bR + \bQ[n]$. To see this, suppose for contradiction that it is not, thus $$q(n) = \sum \theta_i (P_i(n) -P_i(0)) \in \bQ[n]$$ and thus the polynomials $q(n), P_1(n) - P_1(0), \ldots, P_d(n) - P_d(0)$ are linearly dependent over the real numbers and hence over the rationals (as they are all rational polynomials) but $P_1(n) - P_1(0), \ldots, P_d(n) - P_d(0)$ are linearly independent by assumption, so we have a linear combination $$q(n) = \sum_i \theta_i' (P_i(n) - P_i(0))$$ with $\theta_i' \in \bQ$. But by linear independence of $P_1(n) - P_1(0), \ldots, P_d(n) - P_d(0)$ we must have $\theta_i = \theta_i' \in \bQ$.

This means we can apply Weyl's polynomial equidistribution theorem to get that $$ \lim_{N\to \infty} \frac{1}{N} \sum_{n=1}^N e( \langle  P(n), \theta\rangle ) =0 \quad \text{for all } \theta \in \Omega.$$ The dominated convergence theorem now completes the proof. \end{proof}

\begin{proof}[Proof of Theorem~\ref{polynomial mean ergodic thm}] Let $q_0$ be such that $\epsilon_{r,D, Q}(q) < \epsilon$ for all $q \geq q_0$ for an $\epsilon_{r,D,Q}$ as in Proposition~\ref{prop: multivariate Hua bound} and set $q = \text{lcm}\{1, \ldots, q_0\}$. By Lemma~\ref{lemma: polynomial average zero on irrational spectrum}, the left hand side of (\ref{polynomial mean estimate}) remains unchanged if we replace $\mathds{1}_B$ with $P_{\text{Rat}} \mathds{1}_B$. We can write $$P_{\text{Rat}}\mathds{1}_B = \mu(B) + \sum_{\chi \in R^{\ast}_q} c_{\chi} \rho_{\chi} + \sum_{\chi \in \textbf{Rat} \setminus R_q} c_{\chi} \rho_{\chi}  $$ where $\rho_{\chi}$ is a $\chi$-eigenfunction of norm $1$ and $c_{\chi} \in \bC$. From Proposition~\ref{prop: equidist implies small proj} we get that \begin{align}\label{Rq estimate}\left\| \frac{1}{N} \sum_{n =1}^N T^n \sum_{\chi \in R^{\ast}_q} c_{\chi} \rho_{\chi} \right\|^2_2 \leq \left\| \sum_{\chi \in R^{\ast}_q} c_{\chi} \rho_{\chi} \right\|^2_2 \leq (2\delta+\delta^2) \mu(B)^2 \leq 3\delta.\end{align} Now if $\chi \in \textbf{Rat} \setminus R_q$ then the cardinality of the image of $\chi$ is $q' \geq q_0$ and the map $n \mapsto \chi(P(n))$ is $q'$ periodic, hence by Proposition~\ref{prop: multivariate Hua bound} we get that $$ \lim_{N \to \infty} \left| \frac{1}{N} \sum_{n=1}^N \chi(P(n)) \right| = \left| \frac{1}{q'} \sum_{n=1}^{q'} \chi(P(n))  \right| \leq \epsilon. $$ This implies that \begin{align*} \limsup_{N \to \infty} \left\| \frac{1}{N} \sum_{n = 1}^N T^{P(n)} \sum_{\chi \in \textbf{Rat} \setminus R_q} c_{\chi} \rho_{\chi}\right \|^2_2 &=  \limsup_{N \to \infty} \left\| \sum_{\chi \in \textbf{Rat} \setminus R_q}\left(\frac{1}{N} \sum_{n=1}^N \chi(P(n)) \right) c_{\chi} \rho_{\chi}\right \|^2_2 \\ &\leq \epsilon^2  \sum_{\chi \in \textbf{Rat} \setminus R_q} c^2_{\chi} \\&\leq \epsilon^2 \mu(B) \leq \epsilon^2. \end{align*} Finally, combining this estimate with (\ref{Rq estimate}) and using the triangle inequality gives the desired estimate (\ref{polynomial mean estimate}). \end{proof}

The Cauchy-Schwartz inequality and Theorem~\ref{polynomial mean ergodic thm} immediately give the following.

\begin{thm}\label{thm: low complexity and equidistribution gives recurrence} Let $D,r , Q$ be positive integers and $\epsilon > 0$. There exists a positive integer $q = q(D, r, Q, \epsilon)$ such that the following is true: Let $P:\bZ \to \bZ^r$ be a polynomial $(P(n) = (P_1(n), \ldots, P_r(n))$ with degree at most $D$ (that is, $P_i(n) \in \bZ[n]$ with $deg(P_i) \leq D$) such that $P(n)$ has multiplicative complexity $Q$ and $P$ is hyperplane-fleeing. Let $T:\bZ^r \curvearrowright (X,\mu)$ be an ergodic measure preserving system and suppose that $B \subset X$ is $(q,\delta)$-equidistributed. Then there exist arbitrarily large $n \in \bZ$ such that $$\mu(T^{P(n)}B \cap B) > \mu(B)^2 - \epsilon - \sqrt{3\delta}.$$

\end{thm} 

\begin{proof}[Proof of Theorem~\ref{thm: uniform polynomial recurrence}] Take $q = q(D,r,Q,\frac{\epsilon^2}{2})$ as in Theorem~\ref{thm: low complexity and equidistribution gives recurrence}. Now apply the measure increment argument (Lemma~\ref{lemma: measure increment}) to pass to a $T^k$-ergodic component $\nu$ of $\mu$ such that $B$  is $(q, \delta)$-equidistributed where $\delta = \frac{1}{12} \epsilon^4  $ and $k \leq k_0(q, \epsilon) \leq q^{\log(\epsilon^{-1})/\log(1+\delta)}$. So we may apply Theorem~\ref{thm: low complexity and equidistribution gives recurrence} to the $T^k: \bZ^r \curvearrowright (X,\nu)$ and get the desired conclusion (note that $\nu(B') > 0$ implies $\mu(B')>0$). \end{proof}

\section{Constructing polynomials from unipotent elements (proof of Theorem~\ref{thm: uniformly virtually recurrent orbits for unipotents})}
\label{section: from polynomials to unipotent actions}

We now show how to prove Theorem~\ref{thm: uniformly virtually recurrent orbits for unipotents} on uniform recurrence of unipotent actions from Theorem~\ref{thm: uniform polynomial recurrence} on uniform polynomial recurrence. This will amount to constructing appropriate polynomials from sufficiently nice sets of unipotent elements as given in the hypothesis of Theorem~\ref{thm: uniformly virtually recurrent orbits for unipotents}.

\begin{prop} Let $N$, $r$ and $Q$ be positive integers and let $D = D(N,r) = (r+1)^{N+1} - r - 1$. Suppose $u_1, \ldots, u_N \in \operatorname{SL}_r(\bZ)$ are unipotent elements, i.e., $u_j - I$ are nilpotent matrices,  and $v \in \bZ^r$ is such that $$\mathcal{S} = \{u_1^{n_1} \cdots u_N^{n_N} ~|~ n_1, \ldots, n_N \in \bZ \}$$ satisfies the properties that $\mathcal{S}v$ is hyperplane-fleeing and $Q$-coset fleeing in $\bZ^r$. Then there exists a polynomial $P:\bZ \to \bZ^r$ with $P(n) = (P_1(n), \ldots, P_r(n))$ where $P_i(n) \in \bZ[n]$ such that $P(n)$ has degree at most $D$, multiplicative complexity $Q$ and the image $$\{ P(n) ~|~ n \in \bZ\}$$ is contained in $\mathcal{S}v$ and is hyperplane fleeing. \end{prop}

\begin{proof} 

As $u_i$ is unipotent, we have that the entries of the matrix $u_i^{n_i}$ are polynomials of degree at most $r$ in $n_i$. Thus $$S(n_1, \ldots, n_N) = u_1^{n_1} \cdots u_N^{n_N}v = (S_1(n_1, \ldots n_N), \ldots, S_r(n_1, \ldots, n_N))$$ with $S_i(n_1, \ldots, n_N) \in \bZ[n_1, \ldots, n_N]$. Suppose that $\vec{a} = (a_1, \ldots a_r) \in \bZ^r$ and $q \in \bZ$ are such that $gcd(a_1, \ldots, a_r, q) = 1$. Let $b_0 = S(0, \ldots, 0)$ be the constant term of $S$. Now consider the polynomial $$F(n_1, \ldots, n_N) = (S(n_1, \ldots, n_N) - b_0) \cdot \vec{a}$$ and note that it has $0$ constant term. Let $q'$ be the gcd of $q$ and all the coefficients of $F(n_1, \ldots, n_N)$. We claim that $q' \leq Q$. To see this, let $U = \bZ/q\bZ$ and let $\theta: \bZ^r \to U$ be the map given by $\theta(x_1, \ldots, x_r) = \sum_{i=1}^r a_i x_i$ and note that it is surjective since $gcd(a_1, \ldots, a_r, q) = 1$. The image of $\theta \circ S$ is contained in $q'U + \theta(b_0)$, as it is a polynomial which each coefficient dividing $q'$, hence $\mathcal{S}v $ is contained in the coset $\theta^{-1}(q'U) + b_0$, thus from the assumption that it is $Q$-coset fleeing we get that $|\bZ^r /\theta^{-1}(q'U)| \leq Q$. However as $\theta$ is surjective we have that $$|\bZ^r /\theta^{-1}(q'U)| = |U/q'U| = q'$$ where the last equality holds since $q' | q$. Thus $q' \leq Q$ as required.

Now observe that the substitutions $n_j \mapsto n^{(r+1)^j}$ induce a map $\bZ[n_1, \ldots, n_N] \to \bZ[n]$ that is injective on the monomials appearing in $S(n_1, \ldots, n_N)$. Hence $F(n^{r+1}, n^{(r+1)^2}, \ldots, n^{(r+1)^N}) \in \bZ[n]$ is a polynomial in $n$ which has the same set of coefficients as $F(n_1, \ldots , n_N)$. Thus $P(n) = S(n^{r+1}, n^{(r+1)^2}, \ldots, n^{(r+1)^N})$ has degree at most $D = (r+1)^{N+1} - r - 1$ and it has multiplicative complexity $Q$, as required. Finally, $P(n) = (P_1(n), \ldots, P_{r}(n))$ is hyperplane fleeing as otherwise some linear combination of the $P_i(n)$ is a constant function, hence the constant polynomial, hence some non-trivial linear combination of the $S_i(n)$ is constant, contradicting the assumption that $\mathcal{S}v$ is hyperplane-fleeing.
\end{proof}

\section{Applications to the adjoint representation and the proof of Theorem~\ref{thm: uniform virtual recurrence for adjoint representation}}

We now demonstrate how to deduce Theorem~\ref{thm: uniform virtual recurrence for adjoint representation} from Theorem~\ref{thm: uniformly virtually recurrent orbits for unipotents} by showing how the hypothesis of Theorem~\ref{thm: uniformly virtually recurrent orbits for unipotents} is satisfied by the companion matrices in $\mathfrak{sl}_d(\bZ)$. This technique will be easily generalized to $\operatorname{SO}(F)$ for the quadratic form $F(x,y,z) = x^2 - y^2 - z^2$.

Let $\Lambda = \mathfrak{sl}_d({\bZ}) \cong \bZ^{r}$, where $r = d^2 -1$, be the additive group of $d \times d$ integer matrices with zero trace and we let $\Gamma = \operatorname{SL}_d(\bZ)$ act on $\Lambda$ by conjugation.

Note that $\Gamma$ is generated by finitely many unipotents $u_1, \ldots, u_{\ell}$ (the elementary matrices) and hence if we set $$\Gamma_N = \{u_1^{k_1} \cdots u_N^{k_N} ~|~ k_i \in \bZ \}$$ where we use cyclic notation $u_{n} = u_{n \pmod{\ell}}$ then $\Gamma = \bigcup_{N \geq 1} \Gamma_N$ with $\Gamma_N \subset \Gamma_{N+1}$. Note that the image of $u_i$ in $\operatorname{SL}_r(\bZ)$ (i.e., the map $v \mapsto u_i v u_i^{-1}$) is also unipotent since this mapping $\operatorname{SL}_d(\bZ) \to \operatorname{SL}_r(\bZ)$ is a group homomorphism and a polynomial map.

Given a polynomial $p(t) = a_0 + a_1t + \cdots + a_{n-1}t^{n-1} + t^n$, we can define its companion matrix by

$$c_p = \begin{bmatrix}
    0   & \dots & 0 & 0 & -a_0 \\
    1  & \dots & 0 & 0 & -a_1 \\
    \vdots & \ddots & \vdots & \vdots & \vdots \\
    0 & \dots & 1  & 0 & -a_{n-2} \\
    0 & \dots & 0 & 1  & -a_{n-1}
\end{bmatrix}$$

The characteristic polynomial of $c_p$ is $p(t)$. Assume now that $a_{n-1} = 0$. Now consider the elementary matrix

$$\gamma_0 = \begin{bmatrix}
    1   & 0 & 0 & 0 & 1 \\
    0   & 1 & 0 & 0 & 0 \\
    \vdots & \vdots & \ddots & \vdots & \vdots \\
    0 & 0 & 0  & 1 & 0 \\
    0 & 0 & 0 & 0  & 1
\end{bmatrix}$$ and notice that\footnote{Note that $\gamma_0 c_p$ here means conjugation not matrix multiplication, as that is the group action of interest. Note also that $-\delta_{n,2}$ is the Kronecker delta (so $-1$ if $n=2$ and $0$ otherwise)} \begin{align} \label{char polynomials identity} v_0 = \gamma_0 c_p - c_p = \begin{bmatrix}
    0 & \dots  & 0  & 1 & -\delta_{n,2} \\
    0 & \dots  & 0 & 0 & -1 \\
    0 & \dots & 0  & 0 & 0 \\
    \vdots & \dots & \vdots & \vdots & \vdots \\
    0 & \dots & 0 & 0  & 0
\end{bmatrix} \end{align} is a non-zero constant independent of $p$.

\begin{prop}\label{prop: N,Q bounds} There exist constants $Q, N < \infty$ such that for all companion matrices $c_p \in \mathfrak{sl}_d(\bZ)$ the set $\Gamma_Nc_p$ is $Q$-coset fleeing and hyperplane-fleeing.   \end{prop}

\begin{proof} As $v_0 \neq 0$ and the action of $\operatorname{SL}_d(\bZ)$ on the $\bR$-vector space $\mathfrak{sl}_d(\bR)$ is an irreducible representation, we have that the $\bR$-span of the orbit $\Gamma v_0$ is the whole $\mathfrak{sl}_d(\bR)$. It follows that $\mathbb{Z}$-span of $\Gamma v_0$ is a finite index subgroup $W_0$ of $\mathfrak{sl}_d(\bZ)$, say of index $Q$. In fact,  there is a large enough $n$ so that the $\mathbb{Z}$-span of the partial orbit $\Gamma_n v_0$ is $W_0$ (as finitely generated abelian groups are Noetherian). Now take $N \geq n$ such that $\Gamma_n \gamma_0 \subset \Gamma_N$. We claim that these $N$ and $Q$, clearly constructed independently of $c_p$, satisfy the claim. To see this, suppose that $W \leq \Lambda =  \mathfrak{sl}_d({\bZ}) $ is a subgroup such that $\Gamma_N c_p \subset W + a$ for some $a \in \Lambda$. Then $$W \supset \Gamma_Nc_p - \Gamma_Nc_p \supset \Gamma_n(\gamma_0 c_p - c_p) = \Gamma_n v_0 $$ and thus $|\Lambda / W| \leq Q$ as required. Likewise, since the $\mathbb{R}$-span of $\Gamma_N v_0$ is the whole of $\mathfrak{sl}_d(\bR)$, the same argument also gives that $\Gamma_N c_p$ cannot be contained in any translate of a strict subspace, thus showing that $\Gamma_N c_p$ is also hyperplane-fleeing. \end{proof}

\section{Quadratic form $x^2 - y^2 - z^2$ and the completion of the proof of Theorem~\ref{thm: example of uniform virtually recurrent level sets}}
\label{Section7}

 Let $F(x,y,z)=x^2-y^2-z^2$ and observe that $$F(x,y,z)= \operatorname{det} \begin{pmatrix} z & -(x+y) \\ x-y & -z \end{pmatrix}.$$ Hence we may regard $F$ as the determinant map on the abelian subgroup $$\Lambda = \left \{ \begin{pmatrix} a_{11} & a_{12} \\ a_{21} & a_{22} \end{pmatrix} \in \mathfrak{sl}_2(\bZ) \text{ }| \text{ } a_{21} \equiv a_{12} \mod 2 \right\}.$$ Now notice that the conjugation action of  $$\Gamma = \left \langle \begin{pmatrix}1 & 2 \\ 0 & 1 \end{pmatrix}, \begin{pmatrix}1 & 0 \\ 2 & 1 \end{pmatrix} \right \rangle,$$ preserves this additive subgroup and acts irreducibly on $\mathfrak{sl}_2(\bR)$ (see Appendix~\ref{appendix: irreducibility zariski dense}). Moreover, if we let $$\gamma_0 = \begin{pmatrix} 1 & -2 \\ 0 & 1 \end{pmatrix},\quad a_t = \begin{pmatrix} 0 & 2t \\ 2 & 0 \end{pmatrix}, \quad v_0 = \begin{pmatrix} 4 & -8 \\ 0 & -4 \end{pmatrix}$$ then we notice that we have the following analogue of identity (\ref{char polynomials identity}) $$ \gamma_0 a_t - a_t = v_0 \text{ for all } t \in \bZ. $$

Thus, by Theorem~\ref{thm: uniformly virtually recurrent orbits for unipotents}, the same argument as in Proposition~\ref{prop: N,Q bounds} applies to show that the action of $\Gamma$ on $\Lambda$ has uniformly virtually recurrent orbits across the set $\{ a_t ~|~ t \in \bZ\}$ and thus that $F:\bZ^3 \to \bZ$ has uniformly virtually recurrent level sets, as claimed in Theorem~\ref{thm: example of uniform virtually recurrent level sets}.

\section{Quantitative polynomial Bogolyubov's theorem}

We now prove the Polynomial Bogolyubov Theorem (Theorem~\ref{thm: polynomial bog}). By use of Furstenberg's correspondence principle \cite{Fur77} it is enough to show the following result.

\begin{thm}\label{Bog-thm}
Let $\epsilon > 0$, and $R(n) = r_D n^D + \cdots + r_1n \in \bZ[n]$ be a polynomial satisfying $R(0) = 0$ and $D = \deg{R} \geq 2$. There exists a positive integer $k(\epsilon,r_D)$ such that for every ergodic $\bZ^2$ measure-preserving system $(X,\mu,T)$ and any measurable set $B \subset X$ satisfying $\mu(B) > \epsilon$ there exists a positive integer $k \le k(\epsilon,r_D)$ such that for every $m \in \bZ$ we can find $(x,y) \in \bZ^2$ with $x+R(y) = k m$ and 
\[
\mu(B \cap T^{(x,y)}B) > 0.
\]
\end{thm}

%\subsection{Proof of Theorem \ref{Bog-thm}}

\begin{proof} 

Fix such a $B \subset X$ with $\mu(B) > \epsilon$ and let $\delta = \frac{1}{12}\epsilon^4$. We observe that solutions to $x + R(y) = kc$ contain the curve $$P_{k,c}(n) = (kc - R(kn), kn).$$ Note that each $P_{k,c}$ is hyperplane fleeing (as $P_{k,c}(n) \cdot (a_1, a_2)$ is a non-constant polynomial for all $(a_1,a_2) \in \bR^2 \setminus \{(0,0)\}$ as $D \geq 2$) and has multiplicative complexity $r_Dk^{D+1}$, hence $k^{-1}P_{k,c} \in \bZ[n]^2$ has multiplicative complexity $r_Dk^{D}$. In particular, $P_{1,c}(n)$ has multiplicative complexity $r_D$ hence by Theorem~\ref{thm: low complexity and equidistribution gives recurrence} (applied with $Q=r_D$ and $r=2$) there exists a positive integer $q_1 =  q(D, 2, r_D, \frac{\epsilon^2}{2})$ such that \textbf{if it were the case} that $B$ is $(q_1,\delta)$- equidistributed then $$\mu(T^{P_{1,c}(n)}B \cap B) > \mu(B)^2 - \epsilon^2 > 0 \quad \text{for infinitely many } n \in \bZ.$$ Hence the theorem is true with $k=1$ in this case. So now assume that $B$ is \textbf{not} $(q_1,\delta)$- equidistributed. Then by a measure increment argument as given in Lemma~\ref{lemma: measure increment} there exists a $T^{q_1}$ ergodic component, say $\nu$, of $\mu$ such that $\nu(B) \geq (1+\delta) \mu(B)$. By Theorem~\ref{thm: low complexity and equidistribution gives recurrence} now applied to $\nu_1$ there exists an integer $q_2 = q(D,2,r_Dq_1^{D}, \frac{\epsilon^2}{2})$ such that if $B$ is $(q_2,\delta)$ equdistributed with respect to $\nu_1$ then $$\nu_1(\left(T^{q_1}\right)^ {q_1^{-1}P_{q_1,c}(n)}B \cap B) > \nu_1(B)^2 - \epsilon^2 > 0 \quad \text{for infinitely many } n \in \bZ.$$ Note that since $\nu_1 ( B') > 0 \implies \mu(B') > 0$ this means that we are done with $k = q_1$. Hence assume now that $B$ is \textbf{not} $(q_2, \delta)$-equdistributed, thus there exists a $T^{q_2}$-ergodic component, say $\nu_2$ of $\nu_1$ such that $\nu_2(B) \geq (1+\delta) \nu_1(B) \geq (1+\delta)^2 \mu(B)$. Note that $\nu_2$ is a $T^{q_1q_2}$ ergodic component of $\mu$ and so we may repeat the same argument as before with $q_1q_2$ in place of $q_1$ etc. This procedure must eventually stop as after $j$ steps the ergodic component will have $\nu_j (B) \geq (1+\delta)^j \mu(B) \geq (1+\delta)^j \epsilon$ and so the number of steps is bounded as a function of $\epsilon$, as required. The final value of $k$ is then a product of at most $j(\epsilon)$ integers depending only on $\epsilon$ and $r_D$, hence depends only on $\epsilon$ and $r_D$ as required.

\end{proof}

\appendix
\section{Some algebraic facts} \label{appendix: irreducibility zariski dense}

\begin{lemma} Let $\Gamma \leq G \leq \operatorname{GL}_n(\bR)$ be groups such that $G$ is the Zariski closure of $\Gamma$. Suppose that $\rho:G \to \operatorname{GL}_d(\bR)$ is an irreducible representation such that $\rho$ is a polynomial map. Then the restriction $\rho |_{\Gamma}: \Gamma \to \operatorname{GL}_d(\bR)$ is also irreducible. 
\end{lemma}

\begin{proof} Suppose on the contrary that the restriction is reducible. This means that there exists a proper linear subspace $W \leq \bR^d$ and $w \in W$ such that $\rho(\Gamma) w \subset W$. Let $\pi:\bR^d \to \bR^d/W$ denote the quotient map. Then $P:G \to \operatorname{GL}_d(\bR)$ given by $P(g)=\pi (\rho(g) w)$ is a polynomial in $g$ which vanishes for all $g \in \Gamma$. Since $G$ is the Zariski closure of $\Gamma$, we get that $P$ also vanishes on $G$ and hence $\rho(G)w \subset W$, which contradicts the irreducibility of $\rho$.  \end{proof}

\begin{lemma} Let $a,b \in \bZ \setminus \{0\}$ be non-negative integers. Then the subgroup $$\Gamma_0 = \left \langle \begin{pmatrix}1 & a \\ 0 & 1 \end{pmatrix}, \begin{pmatrix}1 & 0 \\ b & 1 \end{pmatrix} \right \rangle$$ is Zariski dense in $\operatorname{SL}_2(\bR)$. 

\end{lemma}

\begin{proof} Let $$U(t) = \begin{pmatrix}1 & t \\ 0 & 1 \end{pmatrix}.$$ We wish to show that the Zariski closure of $\Gamma_0$ contains $U(t)$ and its transpose, for all $t \in \bR$, as these generate $\operatorname{SL}_2(\bR)$. Now suppose that $P:\operatorname{SL}_2(\bR) \to \bR$ is a polynomial map which vanishes on all of $\Gamma_0$. Then, in particular, the polynomial $R: \bR \to \bR$ given by $R(x)=P(U(x))$ vanishes on the infinite set $a\bZ$, and so $R(x)$ is the zero polynomial. Hence $P$ vanishes on $U(t)$, for all $t \in \bR$. This shows that $U(t)$ is in the Zariski closure, and a similar argument applies to its transpose.  \end{proof}

\begin{eg} The adjoint representation $\operatorname{Ad}:\operatorname{SL}_d(\bR) \to \operatorname{GL}(\mathfrak{sl}_d(\bR))$ is a polynomial map. It is an irreducible representation and hence the above may be applied to verify the claims in Theorem~\ref{thm: BFChar on adjoint representation}. \end{eg}

%%% AUTHOR: optional appendix here
%\appendix %% you may comment this out if no Appendix
%\section*{Appendix}
%\section{Improving the constants}
%Material is placed here as needed.

%%% AUTHOR: optional acknowledgments here
\section*{Acknowledgments} %%  you may comment this out if no Ackno
The authors are grateful to the anonymous referee for numerous suggestions that have been incorporated in the final version of the paper.

%%% AUTHOR:
%%% Bibliography goes here. Note that the arXiv cannot process bibtex
%%% or biber bibliographies.  Example of acceptable bibliograpy format:
\bibliographystyle{amsplain}

%% AUTHOR: You can generate such a bibliography from a .bib file by 
%% running pdflatex/bibtex/pdflatex/pdflatex and then pasting the .bbl file
%% between \begin{thebibliography} and \end{bibliography}

%%% AUTHOR: Include a short description of each author following the
%%% structure below. Use the same short tags used previously.  
%%% Use \imageat{} and \imagedot{} instead of "@" and "." in
%%% email addresses-this replaces the symbols with graphics to avoid 
%%% e-mail address harvesting from the .pdf file
\begin{dajauthors}
\begin{authorinfo}[kamil]
  Kamil Bulinski\\
  School of Mathematics and Statistics\\
  University of Sydney, Australia\\
  kamil\imagedot{}bulinski\imageat{}sydney\imagedot{}edu\imagedot{}au\\
 % \url{}
\end{authorinfo}
\begin{authorinfo}[alexander]
  Alexander Fish\\
  School of Mathematics and Statistics\\
  University of Sydney, Australia\\
  alexander\imagedot{}fish\imageat{}sydney\imagedot{}edu\imagedot{}au \\
  \url{https://www.maths.usyd.edu.au/u/afish/}
\end{authorinfo}
\end{dajauthors}

\end{document}